\documentclass{amsart}
\usepackage{amsmath}
\usepackage{amssymb}
\usepackage{amscd}
\usepackage[mathscr]{euscript}
\usepackage{xypic}
\xyoption{all}

\usepackage{url}
\usepackage{listings}
\lstdefinelanguage{GAP}{%
  morekeywords={%
    Assert,Info,IsBound,QUIT,%
    TryNextMethod,Unbind,and,break,%
    continue,do,elif,%
    else,end,false,fi,for,%
    function,if,in,local,%
    mod,not,od,or,%
    quit,rec,repeat,return,%
    then,true,until,while%
  },%
  sensitive,%
  morecomment=[l]\#,%
  morestring=[b]",%
  morestring=[b]',%
}[keywords,comments,strings]
\usepackage[T1]{fontenc}
\usepackage[variablett]{lmodern}
\usepackage{xcolor}
\usepackage{cancel}
\usepackage{hyperref} 
	\hypersetup{backref=true,       
	    pagebackref=true,               
	    hyperindex=true,                
	    colorlinks=true,                
	    breaklinks=true,                
	    urlcolor= black,                
	    linkcolor= black,                
	    bookmarks=true,                 
	    bookmarksopen=false,
	    filecolor=black,
	    citecolor=black,
	    linkbordercolor=red
	}

\title[Commutator length of compact Lie groups]{On the commutator length of compact Lie groups }

\author{Juan Omar G\'omez}
\thanks{}
\address{Fakult\"at f\"ur Mathematik,
Universit\"at Bielefeld, D-33501 Bielefeld, Germany.}
\email{jgomez@math.uni-bielefeld.de}

\author{Victor Torres-Castillo}
\thanks{}
\address{Centro de Investigaci\'on en Matem\'aticas, Unidad M\'erida, PCTY, Carretera Sierra Papacal--Chuburn\'a Puerto Km 5.5, Sierra Papacal, M\'erida, YUC 97302, Mexico.}
\email{victor.torres@cimat.mx}

\author{Bernardo Villarreal}
\thanks{}
\address{Centro de Investigaci\'on en Matem\'aticas, Unidad M\'erida, PCTY, Carretera Sierra Papacal--Chuburn\'a Puerto Km 5.5, Sierra Papacal, M\'erida, YUC 97302, Mexico.}
\email{bernardo.villarreal@cimat.mx}

\newcommand{\comments}[1]{}

\def \N{{\mathbb N}}

\newtheorem{theorem}{Theorem}
\newtheorem{lemma}[theorem]{Lemma}

\newtheorem{proposition}[theorem]{Proposition}
\newtheorem{corollary}[theorem]{Corollary}

\theoremstyle{definition}
\newtheorem{definition}[theorem]{Definition}

\newtheorem{remark}[theorem]{Remark}
 
\keywords{Commutator length, commutator subgroup, compact Lie groups, classifying space for commutativity. }

\subjclass[2020]{Primary: 22E99; Secondary: 55R35.} %55Q05
\thanks{JOG and VT were supported by CONAHCYT under the program `Becas Nacionales de Posgrado'.}
\date{\today}

\AtBeginDocument{\hypersetup{pdfborder={0 0 1}}}

\begin{document}

%------------------------------------  TITLE ------------------------------------%

\maketitle

%------------------------------------ ABSTRACT ------------------------------------%

\begin{abstract}
In this short note we show that the path-connected component of the identity of the derived subgroup of a compact Lie group consists just of commutators.  We also discuss an application of our main result to the homotopy type of the classifying space for commutativity for a compact Lie group whose path-connected component of the identity is abelian.  
\end{abstract}
%---------------TABLE OF CONTENTS -----------------------------------%
%\tableofcontents

%\noindent Version: 
%--------------------------------------------------------------------------------------%
%--------------------------------------------------------------------------------------%
%--------------------------------------------------------------------------------------%
\section*{Introduction}
\label{sec : Introduction}
%--------------------------------------------------------------------------------------%
%-------------------------------------------------------------------------------

Let $G$ be a group, $[G,G]$ its derived subgroup and $x$ an element in $[G,G]$. The \textit{commutator length of $x$} is defined as the minimum number of commutators needed to write $x$ as a product of commutators, and is denoted by $\mathrm{cl}(x)$. It is well known that $\mathrm{cl}(x)$ is not necessarily one. This invariant has been extensively studied in Lie groups, and there are some remarkable results that guaranty that in many cases any element in the derived subgroup is in fact a commutator (see \cite{Dok86}, \cite{Got49}, \cite{PW62}). For instance, from work of M. Goto \cite{Got49} it follows that in a compact connected  Lie group $G$, the derived subgroup of $G$ consists just of commutators. But for non-compact connected Lie groups, the result no longer holds, for example the negative identity matrix in $\mathrm{SL}_2(\mathbb{R})$ is not a commutator. 

In this short note we investigate this invariant for compact Lie groups that are not necessarily connected. Let $G_0$ denote the path-connected component of the identity in $G$.

\begin{theorem}\label{main}
Let $G$ be a compact Lie group. Then any element in $[G,G]_0$ is a commutator, that is, $\mathrm{cl}(x)=1$ for every $x\in [G,G]_0$.
\end{theorem}

We would like to emphasize that in general the group $[G,G]_0$ is not the derived subgroup of a compact connected Lie group, as $G_0$ may not be semi-simple, but it does contain the derived subgroup $[G_0,G_0]$. Moreover, if $G$ is a compact Lie group and $\pi_0(G)$ contains elements of commutative length greater than 1, then so does $G$ (see Remark \ref{Remark sobre clG mayor que uno}). Hence Theorem \ref{main} is the \textit{best} generalization of Goto's result in the compact setting.

Our study of the commutator length of elements in $[G,G]_0$ was mainly motivated by some results in \cite{Villarreal2023} concerning the second homotopy group of the classifying space for commutativity $B(2,G)$ of a compact Lie group $G$.  The space $B(2,G)$ is a variant of the classifying space of $G$ and was introduced in \cite{ACT12}. Similarly as in the theory of principal $G$-bundles, there is universal bundle $E(2,G)\to B(2,G)$ in the sense of \cite[Theorem 2.2]{AG15},  where the total space $E(2,G)$ carries all the failure of $G$ to be commutative, up to homotopy. For instance, when $G$ is a compact Lie group, $E(2,G)$ is a contractible space if and only if $G$ is abelian (see \cite{AGV21}). It is then natural to try to describe the homotopy type of $E(2,G)$ for non-abelian $G$. Our main application of Theorem \ref{main} is the following.

\begin{proposition}\label{prop: app}
    Let $G$ be a compact Lie group in which $G_0$ is abelian. Then up to homotopy, the derived subgroup $[G,G]$ splits off from the loop space $\Omega E(2,G)$.
\end{proposition}

\noindent \textbf{Acknowledgements:} This project was initiated at CIMAT M\'erida’s Algebraic Topology Seminar, in the modality of `mesas de trabajo' that took place in the period January-June 2023.

\section{Preliminaries}

In this note, we will adopt the following convention. Let $g,h$ be elements in a topological group $G$. We write $[g,h]$ to denote the commutator $g^{-1}h^{-1}gh$. As mentioned before we write $G_0$ to denote the path-connected component of the identity of $G$. Recall that the commutator subgroup of a compact Lie group is closed \cite[Theorem 6.11]{HofMor13}, and then coincides with the algebraic commutator subgroup. 

\begin{definition}\label{commutator length}
    Let $G$ be a topological group, and let $g\in [G,G]$. We define the \textit{commutator length} of $g$ as the minimum number of commutators $[g_1,h_1],\ldots,[g_n,h_n]$ such that $g=[g_1,h_1]\cdot\ldots\cdot[g_n,h_n]$, with $g_i,h_i\in G$, and we denote this number by $\mathrm{cl}(g)$.  
    The \textit{commutator length} $\mathrm{cl}(G)$ of $G$ is the supremum in $\N\cup\{\infty\}$ of the set $\{\mathrm{cl}(g)\mid g\in[G,G]\}$, and similarly the \textit{connected commutator length} $\mathrm{cl}(G)_0$ of $G$ is the the supremum of $\{\mathrm{cl}(g)\mid g\in[G,G]_0\}$. 
\end{definition}

Note that under this definition the commutator length of the identity element $1 \in G$ is $\mathrm{cl}(1) = 1$, hence if $G$ is an abelian group then $\mathrm{cl}(G)_0 = \mathrm{cl}(G) = 1$.

\begin{remark}
There is no reason to expect the commutator length of an arbitrary group to be finite, we refer to \cite{Cal08} for some examples of groups with infinite commutator length.  
However, if $G$ is a compact Lie group, then $\mathrm{cl}(G)$ is finite. For example, it follows by Theorem \ref{main} and the fact that $[G,G] = [G,G]_0[F,F]$ for some finite subgroup $F\subset G$ (see Remark \ref{reduccion de extension arbitraria a producto semidirecto} bellow). 
\end{remark}

It is worth highlighting that computing the commutator length of a group is not an easy task, not even for finite groups. In fact, there are some famous conjectures about this invariant. For instance, O. Ore \cite{Ore51} conjectured that finite simple groups have commutator length 1, and D. Dokovik \cite{Dok86} conjectured that simple real Lie groups have commutator length 1. 

\begin{remark}\label{reduccion de extension arbitraria a producto semidirecto}
    Let $G$ be a compact Lie group. Recall that $G$ is the quotient of a semidirect product $\Gamma$ of a connected compact Lie group by a finite subgroup, that is, it has the form $G=(G_0\rtimes F)/H$ where $G_0$ denotes the connected component of the identity, $F$ is finite group and $H$ is a common finite subgroup that is central in $G_0$ but not necessarily in $F$ (see \cite[Lemma 5.1, footnote p. 152]{BorelSerre64}). Let $p\colon \Gamma\to G$ the quotient map. Note that the image of $p|_{[\Gamma,\Gamma]}$ is precisely $[G,G]$. Moreover, it maps $[\Gamma,\Gamma]_0$ to $[G,G]_0$. In particular $\mathrm{cl}(G)\leq \mathrm{cl}(\Gamma)$ and $\mathrm{cl}(G)_0\leq \mathrm{cl}(\Gamma)_0$.
\end{remark}

\section{Extensions of finite groups by tori}

In this section we show that the connected commutator length of an extension of a finite group by a torus is at most 1. In the last section we will use Remark \ref{reduccion de extension arbitraria a producto semidirecto} to extend this result to compact Lie groups. 

First we identify the connected component of the identity of $[G,G]$ when $G$ is a semidirect product of a finite group by a torus. 

\begin{lemma}\label{lemma 1}
Let $G=T\rtimes Q$ be a semidirect product of a torus $T$ by a finite group $Q$. Then $[G,G]_0$ agrees with $[Q,T]$. 
\end{lemma}

\begin{proof}
We have that $$[G,G]=[Q,Q][Q,T][T,T]=[Q,Q][Q,T]\,.$$ Let $g\in [G,G]_0$. By the previous line, we can write $g$ as $rx$ for $r\in [Q,Q]$ and $x\in [Q,T]$. Note that $[G,G]_0\subseteq T$, and since $x\in [Q,T]\subseteq T$, we deduce that $r$ must lie in $T$, hence $r$ is trivial and the result follows.  
\end{proof}

\begin{theorem}\label{thm 1}
Let $G$ be an extension of a finite group $Q$ by a torus $T$. Then 
\[\mathrm{cl}(G)_0 = 1\,.\] 
\end{theorem}

\begin{proof}
We will first prove the case when $G$ is a semidirect product. The general case will follow by Remark \ref{reduccion de extension arbitraria a producto semidirecto}.  Assume that $G$ a semidirect product of $Q$ by $T$.  
Let $TT$ denote the subset consisting of all torsion elements of $T$. Let $x=r[q,t]$  be an element in $[G,G]_0$ with $q\in Q$ and $t\in T$. We claim that if $r\in TT$, then there is a $t'\in T$ such that $x=[q,t']$. Moreover, it is enough to prove this claim for some choice of $t$. Indeed, for every $p\in Q$ we have a continuous group homomorphism $$[p,-]\colon T\to T\,,$$ hence the equality  $r [q,t] = [q,t']$ implies $r = [q,t^{-1}t']$, and with this expression one can readily verify the claim for an arbitrary $t$. Let us assume then that $t$ has infinite order.  Consider a 1-parameter subgroup $S$ of $[G,G]_0$ containing $r[q,t]$. Then $[q,t^N]=(r[q,t])^N$ is in $S$, where $N$ denotes the order of $r$. Since $[q,t^N]$ is of infinite order, it follows that $[q,-]$ must cover $S$ and the claim follows.

On the other hand, let $m,n$ be non-negative integers with $m\leq n$. Define $$P(m,n)=\{[q_1,t_1]\cdot\ldots\cdot[q_n,t_n]\mid q_i\in Q, t_1,\dots t_m \in TT, t_{m+1},\ldots t_n\in T \}\,. $$ 
Note that $P(0,n)$ is a closed subset of $[G,G]_0$ for every $n\geq 1$. In fact, $P(0,|Q|)=[Q,T]=[G,G]_0$ by Lemma \ref{lemma 1}. We have the following inclusion of sets $$P(|Q|-1,|Q|)\subset P(0,|Q|)=[G,G]_0\,.$$ 

Moreover, if $t\in TT$, then $[q,t]\in TT$ for any $q\in Q$. Therefore any element of $P(|Q|-1,|Q|) $ can be written as $r[q,t]$ for some $r\in TT$, $q\in Q$ and $t\in T$. Hence by the previous claim we obtain that $P(|Q|-1,|Q|)\subset P(0,1)$.

Finally, note that the closure of $P(|Q|-1,|Q|)$ must be contained in $P(0,1)$ since the latter is closed. Recall that $TT$ is a dense subspace of $T$. A standard argument of dense subspaces and surjective continuous maps can be used to show that $P(|Q|-1,|Q|)$ is a dense subspace in $[G,G]_0$. Hence $P(0,1)=[G,G]_0$.
\end{proof}

\section{A homotopy splitting of the classifying space for commutativity for extensions of finite groups by tori}\label{aplicaciones}

%%%%%%%punto 1---------------------------

In this section we briefly describe an application and our main motivation to study the connected commutator length, $\mathrm{cl}(G)_0$, for extensions of finite groups by tori.

 A classical result in the homotopy theory of classifying spaces is that the loop space $\Omega BG$ of the classifying space $BG$ of a topological group $G$, is homotopy equivalent to $G$. It has been explored what the analogue result is for a variant of $BG$ called the classifying space for commutativity, denoted $B(2,G)$, which was introduced by A. Adem, F. Cohen and E. Torres-Giese. A model for $B(2,G)$ is the geometric realization of a simplicial space $B_\bullet (2,G)$, where $B_k(2,G)$ is the space of ordered commuting $k$-tuples of $G$ (see \cite{ACT12}). There is a natural map $B(2,G)\to BG$ induced by the inclusions $B_k(2,G)\hookrightarrow G^k$, along which the universal bundle $EG\to BG$ pulls back to a principal $G$-bundle $E(2,G)\to B(2,G)$. Then if $G$ is a Lie group there is a homotopy equivalence 
 \[G\times \Omega E(2,G)\simeq \Omega B(2,G)\] 
 (see \cite[Theorem 6.3 and Remark]{ACT12}). Hence to give a more precise answer on what the homotopy type of $\Omega B(2,G)$ is, we should further study the loop space $\Omega E(2,G)$. 

There is a map $\mathfrak{c}\colon E(2,G)\to B[G,G]$ that can be seen as a \emph{higher} version of the algebraic commutator map $G\times G\to [G,G]$, which may be illustrated by the fact that the restriction of $\mathfrak{c}$ to the 1-skeleton filtration is the suspension of the reduced algebraic commutator $\Sigma G\wedge G\to \Sigma[G,G]$ (see \cite[Section 3]{AGV21} for more details). The map $\mathfrak{c}$ is simply called commutator map. In \cite[Question 21]{AGV21}) the authors asked if for every compact Lie group $G$, the looped commutator map $\Omega\mathfrak{c}\colon \Omega E(2,G)\to [G,G]$ splits, up to homotopy. Theorem \ref{thm 1} together with some previous work gives a positive answer for compact Lie groups in which the path connected component of the identity is abelian. The following corollary is a stronger version of what was stated in Proposition \ref{prop: app}.

\begin{corollary}\label{aplicacion del resultado principal}
    Let $G$ be an extension of a finite group by a torus. Then the looped commutator map $\Omega\mathfrak{c}\colon \Omega E(2,G) \to [G,G]$ splits, up to homotopy.  
\end{corollary}
\begin{proof}
Under the hypothesis of the corollary, \cite[Corollary 3]{Villarreal2023} states that the looped commutator map splits, up to homotopy, as long as $[G,G]_0$ consists of single commutators,  but by Theorem \ref{thm 1} every extension of a finite group by a torus satisfies this condition, hence the result follows. 
\end{proof}

\section{Proof of Theorem \ref{main}}

Let $G$ be a compact Lie group. Recall that there is an extension of groups $$1\to G_0\to G \to \pi_0(G)\to 1\,,$$ where $\pi_0(G)$ is the group of path-connected components of $G$, which is finite. 

\begin{remark}\label{conmutadores en grupos de Lie semisimples compactos y conexos}
    Let $G$ be a  compact connected semi-simple Lie group. As mentioned before, a classical result of Goto states that $\mathrm{cl}(G)=1$. Let $T$ be a maximal torus of $G$. Recall that a theorem of Cartan asserts that any element $g\in G$ can be written as $g = h^{-1}th$ for some $h\in G$ and some $t\in T$. Moreover, since $G$ is semi-simple, if $N_G(T)$ is the normalizer of $T$, then $[N_G(T),N_G(T)]_0=T$. By Theorem \ref{thm 1}, every element of $T$ can be written as a commutator $[n,t]$ for some $t\in T$ and $n\in N_G(T)$. Hence any element of $G$ has the form $h^{-1}[n,t]h$ with $h\in G$, $n\in N_G(T)$ and $t\in T$.  
\end{remark}

Recall that Theorem \ref{main} states that for a compact Lie group $G$,  $\mathrm{cl}(x)=1$ for every $x\in [G,G]_0$. 

\begin{proof}[Proof of Theorem \ref{main}]
Let $F\subset G$ be a finite group as in Remark \ref{reduccion de extension arbitraria a producto semidirecto}.  By \cite[Lemma 5]{Villarreal2023}, we have a decomposition 
\[[G,G]_0 = [F,Z(G_0)_0][G_0,G_0]\,.\] 
 Let $x\in [G,G]_0$. We can write $x=zg$ with $z\in [F,Z(G_0)_0]$ and $g\in [G_0,G_0]$.  Consider the group extension 
 \[1\to Z(G_0)_0\to Z(G_0)_0 F\to F/(F\cap Z(G_0)_0)\to 1\,.\]
  By Theorem \ref{thm 1} we can write $z$ as a commutator $[q,z^\prime]$, for some $q\in F$ and $z^\prime\in Z(G_0)_0$. By \cite{BM55} every automorphism of finite order of $[G_0,G_0]$ has an invariant maximal torus, that is, $q^{-1}Tq=T$ for a maximal torus $T\subset [G_0,G_0]$. Now since $[G_0,G_0]$ is a compact connected semi-simple Lie group and all maximal tori in $[G_0,G_0]$ are conjugate, by Remark \ref{conmutadores en grupos de Lie semisimples compactos y conexos} we may further assume that $g=h^{-1}[n,t]h$ for some $h\in [G_0,G_0]$, $n\in N_{[G_0,G_0]}(T)$ and $t\in T$, where $N_{[G_0,G_0]}(T)$ is the normalizer of $T$.  Moreover, since $z\in Z(G_0)$ we can write $x=h^{-1}[q,z^\prime][n,t]h$. 
  
 Note that \[[q,z^\prime][n,t]\in [\langle q\rangle,Z(G_0)_0][N_{[G_0,G_0]}(T),N_{[G_0,G_0]}(T)]_0\]
 which is a subgroup of the path-connected component of the commutator subgroup $[N_{G_0}(T_0)\langle q\rangle,N_{G_0}(T_0)\langle q\rangle]$, where $T_0=Z(G_0)_0T$ is a maximal torus of $G_0$. We can now consider the group extension 
 \[1\to T_0\to N_{G_0}(T_{0}) \langle q\rangle\to Q\to 1\,,\] 
 where $Q$ is a finite group as $q$ leaves $N_{G_0}(T_0)$ invariant, as well.  Invoking again Theorem \ref{thm 1}, the element $[q,z^\prime][n,t]$ is a commutator and hence so is $x$. 
\end{proof}

We finish this note with a sufficient condition for compact Lie groups to have commutator length 1.

\begin{corollary}
Let $G$ be a compact Lie group. If the projection $G\to \pi_0(G)$ admits a section and $\pi_0(G)$ is abelian, then $\mathrm{cl}(G)=1$.
\end{corollary}

\begin{proof}
 Note that under these hypotheses $[G,G]$ is connected. Thus $\mathrm{cl}(G)_0=\mathrm{cl}(G)$. Therefore the result follows by Theorem \ref{main}. 
\end{proof}

\begin{remark}\label{Remark sobre clG mayor que uno}
For a compact Lie group $G$, we have that $\mathrm{cl}(\pi_0(G))$ is a lower bound for $\mathrm{cl}(G)$. Hence we can construct compact Lie groups whose derived subgroups do not consist only of commutators, see \cite{Isa77} for an example of a family of finite groups with this property. 
\end{remark}

\end{document}